\documentclass[secthm,seceqn,amsthm,ussrhead,reqno, 12pt]{amsart}
\usepackage[utf8]{inputenc}
\usepackage[english]{babel}
\usepackage[symbol]{footmisc}
\usepackage{amssymb,amsmath,amsthm,amsfonts,xcolor,enumerate,hyperref,comment,longtable,cleveref}

\usepackage{times}
\usepackage{cite}
\usepackage{pdflscape}
\usepackage{ulem}
\usepackage[mathcal]{euscript}
\usepackage{tikz}
\usepackage{hyperref}
\usepackage{cancel}
\usepackage{stmaryrd}

\newtheorem*{theorem}{Theorem}

\newtheorem{thrm}{Theorem} 
\newtheorem{cor}[thrm]{Corollary}
\newtheorem{lem}[thrm]{Lemma}
\newtheorem{prop}[thrm]{Proposition}

\theoremstyle{definition}
\newtheorem{defn}[thrm]{Definition}

\crefrangeformat{equation}{#3(#1)#4--#5(#2)#6}

\crefname{thrm}{Theorem}{Theorems}
\crefname{lem}{Lemma}{Lemmas}
\crefname{cor}{Corollary}{Corollaries}
\crefname{prop}{Proposition}{Propositions}
\crefname{defn}{Definition}{Definitions}
\crefname{exm}{Example}{Examples}
\crefname{rem}{Remark}{Remarks}
\crefname{section}{Section}{Sections}
\crefname{equation}{\unskip}{\unskip}
\crefname{enumi}{\unskip}{\unskip}

\renewcommand{\iff}{\Leftrightarrow}
\newcommand{\impl}{\Rightarrow}

\newcommand{\af}{\alpha}
\newcommand{\bt}{\beta}
\newcommand{\gm}{\gamma}

\newcommand{\Dl}{\Delta}
\newcommand{\G}{\Gamma}

\newcommand{\vf}{\varphi}

\newcommand{\C}{\mathbb{C}}
\newcommand{\Z}{\mathbb{Z}}

\newcommand{\sst}{\subseteq}

\newcommand{\m}{{}^{-1}}

\begin{document}

	\noindent{\Large  
		Transposed Poisson structures on Witt type algebras} 

 \bigskip

 \bigskip

 {\bf
	Ivan Kaygorodov\footnote{CMA-UBI, Universidade da Beira Interior, Covilh\~{a}, Portugal; \    kaygorodov.ivan@gmail.com}	\&
Mykola Khrypchenko\footnote{ Departamento de Matem\'atica, Universidade Federal de Santa Catarina,     Brazil; \ nskhripchenko@gmail.com}
	}
	\

	\medskip

	\ 
	
	\noindent {\bf Abstract:} {\it 	
We describe $\frac{1}{2}$-derivations, and hence transposed Poisson algebra structures, on Witt type Lie algebras $V(f)$, where $f:\G\to\C$ is non-trivial and $f(0)=0$. More precisely, if $|f(\G)|\ge 4$, then all the transposed Poisson algebra structures on $V(f)$ are mutations of the group algebra structure $(V(f),\cdot)$ on $V(f)$. If $|f(\G)|=3$, then we obtain the direct sum of $3$ subspaces of $V(f)$, corresponding to cosets of $\G_0$ in $\G$, with multiplications being different mutations of $\cdot$. The case $|f(\G)|=2$ is more complicated, but also deals with certain mutations of $\cdot$. As a consequence, new Lie algebras that admit non-trivial ${\rm Hom}$-Lie algebra structures are found. 	
	}

	\
	

	
	\tableofcontents

	\section*{Introduction} 
	Poisson algebras originated from the Poisson geometry in the 1970s and have shown their importance in several areas of mathematics and physics, such as Poisson manifolds, algebraic geometry, operads, quantization theory, quantum groups, and classical and quantum mechanics. One of the popular topics in the theory of Poisson algebras is the study of all possible Poisson algebra structures with fixed Lie or associative part~\cite{jawo,said2,YYZ07,kk21}.
	Recently, Bai, Bai, Guo, and Wu~\cite{bai20} have introduced a dual notion of the Poisson algebra, called \textit{transposed Poisson algebra}, by exchanging the roles of the two binary operations in the Leibniz rule defining the Poisson algebra. 
	They have shown that a transposed Poisson algebra defined this way not only shares common properties of a Poisson algebra, including the closedness under tensor products and the Koszul self-duality as an operad, but also admits a rich class of identities. More significantly, a transposed Poisson algebra naturally arises from a Novikov-Poisson algebra by taking the commutator Lie algebra of the Novikov algebra.
	Thanks to \cite{bfk22}, 
	any unital transposed Poisson algebra is
	a particular case of a ``contact bracket'' algebra 
	and a quasi-Poisson algebra.
	Later, in a recent paper by Ferreira, Kaygorodov, and  Lopatkin
	a relation between $\frac{1}{2}$-derivations of Lie algebras and 
	transposed Poisson algebras has been established \cite{FKL}. 	These ideas were used to describe all transposed Poisson structures 
	on  Witt and Virasoro algebras in  \cite{FKL};
	on   twisted Heisenberg-Virasoro,   Schr\"odinger-Virasoro  and  
	extended Schr\"odinger-Virasoro algebras in \cite{yh21};
	on   oscillator algebras in  \cite{bfk22}. 
	It was proved that each complex finite-dimensional solvable Lie algebra has a non-trivial transposed Poisson structure \cite{klv22}.
		The ${\rm Hom}$- and ${\rm BiHom}$-versions of transposed Poisson algebras and
		transposed Poisson bialgebras have been considered in \cite{hom, bihom, bl22}.		For the list of actual open questions on transposed Poisson algebras see \cite{bfk22}.
	
	The first non-trivial example of a transposed Poisson algebra was constructed on the Witt algebra with the multiplication 
	law $[e_i,e_j]=(i-j)e_{i+j}$ for $i,j \in \mathbb  Z$ (see, \cite{FKL}).
	This attracted certain interest to the description of transposed Poisson structures on Lie algebras related to the Witt algebra.
	Thus, all transposed Poisson structures   on 
	the Virasoro algebra  \cite{FKL},
	Block-type Lie algebras   and  Block-type Lie superalgebras  \cite{kk22}
	have been described.
	In the last years the concept of Witt type Lie algebra has been enlarged and generalized by various authors, such as Kawamoto,  Osborn, Đoković, Zhao,  Xu,  Passman,  Jordan, etc. (see, for example, \cite{witt} and references therein).
	In the present paper, we study transposed Poisson structures on the class of
	Witt type Lie algebras $V(f)$ defined by Yu in \cite{Yu97}.
    They are playing a critical role in the classification of simple Lie algebras on a lattice \cite{KM13}. 
	Lie admissible structures on Witt type algebras have been described in \cite{BC11}.

	 The first main part of our work is devoted to a description of $\frac{1}{2}$-derivations of $V(f)$, where $f:\G\to\C$ is nontrivial and $f(0)=0$. This is done in \cref{sec-halfder}. We first prove several lemmas that hold for an arbitrary $f$. Then, based on these lemmas and \cite[Lemmas 4.4 and 4.6]{Yu97}, we split our description into $3$ different cases: $|f(\G)|\ge 4$, $|f(\G)|=3$ and $|f(\G)|=2$, leading to 
 \cref{half-der-|f(G)|>=4,half-der-|f(G)|=3,half-der-|f(G)|=2}. We then proceed with the study of transposed Poisson algebra structures $(V(f),*)$ on $V(f)$ in \cref{sec-tp}. It is also divided into $3$ parts corresponding to $|f(\G)|\ge 4$, $|f(\G)|=3$ and $|f(\G)|=2$. The full description of $(V(f),*)$ is given in \cref{tp-4,tp-3,tp-2}, respectively. We gather these results into one main theorem.
\begin{theorem}
Let $\G$ be an abelian group and $f:\G\to\C$ a function satisfying $f(0)=0$ and
\begin{align*}
    \big(f(\af+\bt)-f(\af)-f(\bt)\big)\big(f(\af)-f(\bt)\big)=0.
\end{align*}
Let $\G_0=f^{-1}(0)$.
\begin{enumerate}
    \item If $|f(\G)|\ge 4$, then transposed Poisson algebra structures on $V(f)$ are exactly mutations $(V(f),\cdot_b)$ of $e_\af\cdot e_\bt=e_{\af+\bt}$, where $\af,\bt\in\G$. 

    \item If $|f(\G)|=3$, then transposed Poisson algebra structures on $V(f)$ are exactly of the form $\bigoplus_{i=0}^2 (V(f)_i,\cdot_{b_i})$ for some $b_i\in V(f)_j$ with $i+j\equiv 0 \,\mathrm{mod}\,3$. Here $V(f)_i$, $i=0,1,2$, are the subspaces of $V(f)$ corresponding to the cosets of $\G_0$ in $\G$.

    \item If $|f(\G)|=2$, then transposed Poisson algebra structures on $V(f)$ are exactly of the form 
    \begin{align*}
        e_\af*e_\bt=
		\begin{cases}
			e_\af\cdot_b e_\bt, & \af,\bt\in\G_0,\\
			e_\af\cdot_{b^0} e_\bt, & \af\in\G_0,\bt\not\in\G_0\text{ or }\af\not\in\G_0,\bt\in\G_0,\\
			0, & \af,\bt\not\in\G_0,
		\end{cases}
    \end{align*}
    where $b=\sum_{\gm\in\G}b_\gm e_\gm\in V(f)$ and $b^0=\sum_{\gm\in\G_0}b_\gm e_\gm\in V(f)$.  
\end{enumerate}
\end{theorem}

	\section{Preliminaries}\label{prelim}
 	
	All the algebras below will be over $\mathbb C$ and all the linear maps will be $\mathbb C$-linear, unless otherwise stated.

	\subsection{Transposed Poisson algebras}	
	\begin{defn}\label{tpa}
		Let ${\mathfrak L}$ be a vector space equipped with two nonzero bilinear operations $\cdot$ and $[\cdot,\cdot].$
		The triple $({\mathfrak L},\cdot,[\cdot,\cdot])$ is called a \textit{transposed Poisson algebra} if $({\mathfrak L},\cdot)$ is a commutative associative algebra and
		$({\mathfrak L},[\cdot,\cdot])$ is a Lie algebra that satisfies the following compatibility condition
		\begin{align}\label{Trans-Leibniz}
		2z\cdot [x,y]=[z\cdot x,y]+[x,z\cdot y].
		\end{align}
	\end{defn}
	
	Transposed Poisson algebras were first introduced in a paper by Bai, Bai, Guo, and Wu \cite{bai20}.
	
	\begin{defn}\label{tp-structures}
		Let $({\mathfrak L},[\cdot,\cdot])$ be a Lie algebra. A \textit{transposed Poisson algebra structure} on $({\mathfrak L},[\cdot,\cdot])$ is a commutative associative multiplication $\cdot$ on $\mathfrak L$ which makes $({\mathfrak L},\cdot,[\cdot,\cdot])$ a transposed Poisson algebra.
	\end{defn}

	\begin{defn}\label{12der}
		Let $({\mathfrak L}, [\cdot,\cdot])$ be an algebra and $\varphi:\mathfrak L\to\mathfrak L$ a linear map.
		Then $\varphi$ is a \textit{$\frac{1}{2}$-derivation} if it satisfies
		\begin{align}\label{vf(xy)=half(vf(x)y+xvf(y))}
		\varphi \big([x,y]\big)= \frac{1}{2} \big([\varphi(x),y]+ [x, \varphi(y)] \big).
		\end{align}
	\end{defn}
	Observe that $\frac{1}{2}$-derivations are a particular case of $\delta$-derivations introduced by Filippov in \cite{fil1}
	(see also \cite{k12,z10} and references therein). The space of all $\frac{1}{2}$-derivations of an algebra $\mathfrak L$ will be denoted by $\Dl(\mathfrak L).$

	\cref{tpa,12der} immediately imply the following key Lemma.
	\begin{lem}\label{glavlem}
		Let $({\mathfrak L},[\cdot,\cdot])$ be a Lie algebra and $\cdot$ a new binary (bilinear) operation on ${\mathfrak L}$. Then $({\mathfrak L},\cdot,[\cdot,\cdot])$ is a transposed Poisson algebra 
		if and only if $\cdot$ is commutative and associative and for every $z\in{\mathfrak L}$ the multiplication by $z$ in $({\mathfrak L},\cdot)$ is a $\frac{1}{2}$-derivation of $({\mathfrak L}, [\cdot,\cdot]).$
	\end{lem}
	
	The basic example of a $\frac{1}{2}$-derivation is the multiplication by a field element.
	Such $\frac{1}{2}$-derivations will be called \textit{trivial}.
	
	\begin{thrm}\label{princth}
		Let ${\mathfrak L}$ be a Lie algebra without non-trivial $\frac{1}{2}$-derivations.
		Then all transposed Poisson algebra structures on ${\mathfrak L}$ are trivial.
	\end{thrm}
	
	Let us recall the definition of ${\rm Hom}$-structures on Lie algebras.
	\begin{defn}
		Let $({\mathfrak L}, [\cdot,\cdot])$ be a Lie algebra and $\varphi$ be a linear map.
		Then $({\mathfrak L}, [\cdot,\cdot], \varphi)$ is a ${\rm Hom}$-Lie structure on $({\mathfrak L}, [\cdot,\cdot])$ if 
		\[
		[\varphi(x),[y,z]]+[\varphi(y),[z,y]]+[\varphi(z),[x,y]]=0.
		\]
	\end{defn}
Filippov proved that each nonzero $\delta$-derivation ($\delta\neq0,1$) of a Lie algebra gives a non-trivial ${\rm Hom}$-Lie algebra structure \cite[Theorem 1]{fil1}.

	

	\subsection{Witt type algebras}

	\begin{defn}
		Let $\G$ be an abelian group and $f:\G\to\C$ a function. The \textit{Witt type algebra} $V(f)$ is a vector space with basis $\{e_\af\}_{\af\in\G}$ and multiplication
		\begin{align}\label{[e_af_e_bt]=(f(bt)-f(af))e_(af+bt)}
			[e_\af,e_\bt]=(f(\bt)-f(\af))e_{\af+\bt}.
		\end{align}
	\end{defn}
	 Replacing $f$ by $f-f(0)$, we obtain
	\begin{align}\label{f(0)=zero}
		f(0)=0
	\end{align}
	without changing \cref{[e_af_e_bt]=(f(bt)-f(af))e_(af+bt)}.	Then $V(f)$ is a Lie algebra if and only if
	\begin{align}\label{(f(af+bt)-f(af)-f(bt))(f(af)-f(bt))=0}
		\big(f(\af+\bt)-f(\af)-f(\bt)\big)\big(f(\af)-f(\bt)\big)=0
	\end{align}
	for all $\af,\bt\in\G$. Hence, if $f(\af)\ne f(\bt)$, then $f(\af+\bt)=f(\af)+f(\bt)$. In particular, $f(-\af)\ne f(\af)$ implies $f(-\af)=-f(\af)$. Thus, either $f(-\af)=f(\af)$ or $f(-\af)=-f(\af)$.
	
	In what follows, we assume \cref{f(0)=zero,(f(af+bt)-f(af)-f(bt))(f(af)-f(bt))=0}. It is also natural to assume that $|f(\G)|\ge 2$, since otherwise $V(f)$ is abelian. Let $\G_0=f\m(0)$. The next result is due to Yu~\cite{Yu97}.
	\begin{lem}\label{G_0-subgroup}
			The set $\G_0$ is a subgroup of $\G$, and $\af-\bt\in\G_0\impl f(\af)=f(\bt)$ for all $\af,\bt\in\G$.
	\end{lem}



		\section{Transposed Poisson structures on Witt type algebras}

	\subsection{$\frac 12$-derivations of $V(f)$}\label{sec-halfder}
	
	Observe that $V(f)=\oplus_{\af\in\G} \C e_\af$ is a $\G$-grading. So, any linear map $\vf:V(f)\to V(f)$ decomposes as
	\begin{align*}
	\vf=\sum_{\gm\in\G}\vf_\gm,
	\end{align*}
	where $\vf_\gm:V(f)\to V(f)$ is a linear map such that $\vf_\gm(\C e_\af)\sst \C e_{\af+\gm}$ for all $\af\in\G$. It follows that $\vf$ is a $\frac 12$-derivation of $V(f)$ if and only if $\vf_\gm$ is a $\frac 12$-derivation of $V(f)$ for all $\gm\in\G$. Let us write 
	\begin{align}\label{vf_af(e_bt)=d_af(bt)e_(af+bt)}
	\vf_\gm(e_\af)=d_\gm(\af)e_{\af+\gm},
	\end{align}
	where $d_\gm:\G\to\C$.
	
	\begin{lem}\label{conditions-on-d}
		Let $\vf_\gm:V(f)\to V(f)$ be a linear map satisfying \cref{vf_af(e_bt)=d_af(bt)e_(af+bt)}. Then $\vf_\gm\in \Dl (V(f))$ if and only if 
		\begin{align}\label{one-half-der-in-terms-of-d_af}
		2(f(\bt)-f(\af))d_\gm(\af+\bt)=(f(\bt)-f(\af+\gm))d_\gm(\af)+(f(\bt+\gm)-f(\af))d_\gm(\bt).
		\end{align}
	\end{lem}
	\begin{proof}
		We have
		\begin{align*}
		2\vf_\gm([e_\af,e_\bt])&= 2\vf_\gm((f(\bt)-f(\af))e_{\af+\bt})=2(f(\bt)-f(\af))d_\gm(\af+\bt)e_{\af+\bt+\gm}
		\end{align*}
		and
		\begin{align*}
		[\vf_\gm(e_\af),e_\bt]+[e_\af,\vf_\gm(e_\bt)]&=[d_\gm(\af)e_{\af+\gm},e_\bt]+[e_\af,d_\gm(\bt)e_{\bt+\gm}]\\
		&=(f(\bt)-f(\af+\gm))d_\gm(\af)e_{\af+\bt+\gm}+(f(\bt+\gm)-f(\af))d_\gm(\bt)e_{\af+\bt+\gm}.
		\end{align*}
	\end{proof}
	
	\begin{lem}\label{d_gm(af)=d_gm(0)}
		Let $\vf_\gm\in\Dl(V(f))$ satisfying \cref{vf_af(e_bt)=d_af(bt)e_(af+bt)}. If $f(\af)\ne f(\gm)$, then $d_\gm(\af)=d_\gm(0)$. If $f(\af)=f(\gm)$ and $f(\af+\gm)\ne 2 f(\gm)$, then $d_\gm(\af)=0$.
	\end{lem}
\begin{proof}
	Let us use \cref{one-half-der-in-terms-of-d_af} with $\bt=0$:
	\begin{align*}
		2(f(0)-f(\af))d_\gm(\af)=(f(0)-f(\af+\gm))d_\gm(\af)+(f(\gm)-f(\af))d_\gm(0).
	\end{align*}
	Then, thanks to \cref{f(0)=zero}, we have
	\begin{align}\label{(f(af+gm)-2f(af))d_gm(af)=(f(gm)-f(af))d_gm(0)}
		(f(\af+\gm)-2f(\af))d_\gm(\af)=(f(\gm)-f(\af))d_\gm(0).
	\end{align}
	If $f(\af)\ne f(\gm)$, then $f(\af+\gm)=f(\af)+f(\gm)$ by \cref{(f(af+bt)-f(af)-f(bt))(f(af)-f(bt))=0}. Then dividing both sides of \cref{(f(af+gm)-2f(af))d_gm(af)=(f(gm)-f(af))d_gm(0)} by $f(\gm)-f(\af)\ne 0$ we obtain $d_\gm(\af)=d_\gm(0)$. If $f(\af)=f(\gm)$ and $f(\af+\gm)\ne 2 f(\af)$, then \cref{(f(af+gm)-2f(af))d_gm(af)=(f(gm)-f(af))d_gm(0)} implies $d_\gm(\af)=0$.
\end{proof}

\begin{lem}\label{d_gm(af)=d_gm(0)-when-there-is-bt}
	Let $\vf_\gm\in\Dl(V(f))$ satisfying \cref{vf_af(e_bt)=d_af(bt)e_(af+bt)}. If $f(\af)=f(\gm)$, $f(\af+\gm)=2f(\gm)$ and there exists $\bt\in\G$ with $f(\bt)\not\in \{0,f(\gm),2f(\gm)\}$, then $d_\gm(\af)=d_\gm(0)$.
\end{lem}
\begin{proof}
	It follows from \cref{one-half-der-in-terms-of-d_af} that
	\begin{align}\label{one-half-der-in-terms-of-d_af-f(af)=f(gm)-and-f(af+bt)=2f(af)}
		2(f(\bt)-f(\gm))d_\gm(\af+\bt)=(f(\bt)-2 f(\gm))d_\gm(\af)+(f(\bt+\gm)-f(\gm))d_\gm(\bt).
	\end{align}
	Since $f(\bt)\ne f(\gm)$, then $f(\bt+\gm)=f(\bt)+f(\gm)$ by \cref{(f(af+bt)-f(af)-f(bt))(f(af)-f(bt))=0}. Moreover, $d_\gm(\bt)=d_\gm(0)$ by \cref{d_gm(af)=d_gm(0)}. Now, $f(\bt)\ne f(\af)$ implies $f(\af+\bt)=f(\af)+f(\bt)$, which is different from $f(\af)=f(\gm)$, because $f(\bt)\ne 0$. Hence, $d_\gm(\af+\bt)=d_\gm(0)$ by \cref{d_gm(af)=d_gm(0)}. Thus, \cref{one-half-der-in-terms-of-d_af-f(af)=f(gm)-and-f(af+bt)=2f(af)} becomes
	\begin{align*}
		(f(\bt)-2 f(\gm))(d_\gm(\af)-d_\gm(0))=0.
	\end{align*}
	Finally, using $f(\bt)\ne 2f(\gm)$, we get the result.
\end{proof}

\begin{lem}\label{d_gm(af)=d_gm(0)-if-gm-in-G_0}
	Let $|f(\G)|\ge 2$ and $\vf_\gm\in\Dl(V(f))$ satisfying \cref{vf_af(e_bt)=d_af(bt)e_(af+bt)}. If $\gm\in\G_0$, then $d_\gm(\af)=d_\gm(0)$ for all $\af\in\G$.
\end{lem}
\begin{proof}
	If $f(\af)\ne 0$, then $d_\gm(\af)=d_\gm(0)$ by \cref{d_gm(af)=d_gm(0)}. If $f(\af)=0$, then $f(\af+\gm)=0=2f(\gm)$ by \cref{G_0-subgroup}. But $\{0,f(\gm),2f(\gm)\}=\{0\}$ and $|f(\G)|\ge 2$, so there exists $\bt\in\G$ with $f(\bt)\not\in \{0,f(\gm),2f(\gm)\}$. Hence, $d_\gm(\af)=d_\gm(0)$ thanks to \cref{d_gm(af)=d_gm(0)-when-there-is-bt}. 
\end{proof}

\begin{lem}\label{vf(e_af)=ae_(af+gm)-gm-in-G_0}
	For any fixed $\gm\in\G_0$ and $a\in\C$ the linear map $\vf:V(f)\to V(f)$ given by
	\begin{align*}
		\vf(e_\af)=ae_{\af+\gm}
	\end{align*}
	is a $\frac 12$-derivation of $V(f)$.
\end{lem}
\begin{proof}
	We have $\vf=\vf_\gm$, where $d_\gm(\af)=a$ for all $\af\in\G$. Take arbitrary $\af,\bt\in\G$. 
	
	\textit{Case 1.} $\af,\bt\in\G_0$. Then  $\af+\gm,\bt+\gm\in\G_0$ by \cref{G_0-subgroup}, so both sides of \cref{one-half-der-in-terms-of-d_af} are zero.
	
	\textit{Case 2.} $\af\in\G_0$, $\bt\not\in\G_0$. Then $f(\af+\gm)=0$, $f(\bt)\ne f(\gm)=f(\af)$, so 
    \begin{align*}
    f(\bt+\gm)=f(\bt)+f(\gm)=f(\bt).
    \end{align*} 
    It follows that both sides of \cref{one-half-der-in-terms-of-d_af} are equal to $2f(\bt)a$.
	
	\textit{Case 3.} $\af\not\in\G_0$, $\bt\in\G_0$. Similarly to Case~2, we see that both sides of \cref{one-half-der-in-terms-of-d_af} are equal to $-2f(\af)a$.
	
	\textit{Case 4.} $\af,\bt\not\in\G_0$. Then $f(\af+\gm)=f(\af)+f(\gm)=f(\af)$ and $f(\bt+\gm)=f(\bt)+f(\gm)=f(\bt)$, so both sides of \cref{one-half-der-in-terms-of-d_af} are equal to $2(f(\bt)-f(\af))a$.
	%
	%
\end{proof}	

\subsubsection{The case $|f(\G)|\ge 4$}

By \cite[Lemma 4.6]{Yu97} the function $f$ is additive in this case.

\begin{lem}\label{d_gm(af)=d_gm(0)-for-|f(G)|>=4}
	Let $|f(\G)|\ge 4$ and $\vf_\gm\in\Dl(V(f))$ satisfying \cref{vf_af(e_bt)=d_af(bt)e_(af+bt)}. Then $d_\gm(\af)=d_\gm(0)$ for all $\gm,\af\in\G$.
\end{lem}
\begin{proof}
	Since $f$ is additive, $f(\af)=f(\gm)$ implies $f(\af+\gm)=2f(\gm)$. Moreover, since $|f(\G)|\ge 4$, we can always find $\bt\in\G$ with $f(\bt)\not\in \{0,f(\gm),2f(\gm)\}$. So, the result follows from \cref{d_gm(af)=d_gm(0),d_gm(af)=d_gm(0)-when-there-is-bt}.
\end{proof}

\begin{lem}\label{vf(e_af)=ae_(af+gm)-half-der}
	Let $|f(\G)|\ge 4$. Then for any fixed $\gm\in\G$ and $a\in\C$ the linear map $\vf:V(f)\to V(f)$ given by
	\begin{align}
		\vf(e_\af)=ae_{\af+\gm}
	\end{align}
	is a $\frac 12$-derivation of $V(f)$.
\end{lem}
\begin{proof}
	We have $\vf=\vf_\gm$ with $d_\gm(\af)=a$ for all $\af\in\G$. Let us check \cref{one-half-der-in-terms-of-d_af} for arbitrary $\af,\bt\in\G$. Since $d_\gm(\af+\bt)=d_\gm(\af)=d_\gm(\bt)=a$ and $f(\af+\gm)=f(\af)+f(\gm)$, $f(\bt+\gm)=f(\bt)+f(\gm)$, both sides of \cref{one-half-der-in-terms-of-d_af} are equal to $2(f(\bt)-f(\af))a$.
\end{proof}

\begin{prop}\label{half-der-|f(G)|>=4}
	Let $|f(\G)|\ge 4$. Then $\frac 12$-derivations of $V(f)$ are of the form
	\begin{align}
		\vf(e_\af)=\sum_{\gm\in\G} a_\gm e_{\af+\gm},
	\end{align} 
where $\{a_\gm\}_{\gm\in\G}\sst\C$ is a sequence having only a finite number of non-zero elements.
\end{prop} 
\begin{proof}
	A consequence of \cref{d_gm(af)=d_gm(0)-for-|f(G)|>=4,vf(e_af)=ae_(af+gm)-half-der}.
\end{proof}

\subsubsection{The case $|f(\G)|=3$}

By \cite[Lemma 4.4]{Yu97} there are a surjective group homomorphism $\tau:\G\to\Z/3\Z$ and $c\in\C^*$ such that $f(\af)=\pm c\iff\tau(\af)=\pm 1$ for all $\af\in\G$. 

\begin{lem}\label{d_gm(af)=d_gm(0)-for-|f(G)|=3}
	Let $|f(\G)|=3$, $\gm\not\in\G_0$ and $\vf_\gm\in\Dl(V(f))$ satisfying \cref{vf_af(e_bt)=d_af(bt)e_(af+bt)}. Then $d_\gm(\af)=0$ for all $\af\in\G$.
\end{lem}
\begin{proof}
	If $f(\af)=f(\gm)=c$, then $\tau(\af)=\tau(\gm)=1$, so $\tau(\af+\gm)=-1$, whence $f(\af+\gm)=-c$. Since $-c\ne 2c=2f(\gm)$, we conclude that $d_\gm(\af)=0$ by \cref{d_gm(af)=d_gm(0)}. 
	
	The case $f(\af)=f(\gm)=-c$ is analogous. 
	
	If $f(\af)\ne f(\gm)$, then $d_\gm(\af)=d_\gm(0)$ by \cref{d_gm(af)=d_gm(0)}, so it remains to prove that $d_\gm(0)=0$. To this end, choose $\af,\bt\in\G$ with $f(\af)=f(\gm)$ and $f(\bt)=-f(\gm)$ (this is possible in view of \cite[Lemma 4.4]{Yu97}). Then $d_\gm(\af)=0$ as we have just proved and $d_\gm(\bt)=d_\gm(0)$ by \cref{d_gm(af)=d_gm(0)}. Moreover, $f(\bt)\ne f(\af)=f(\gm)$ yields $f(\bt+\gm)=f(\bt)+f(\gm)=0$ and $f(\af+\bt)=f(\af)+f(\bt)=0\ne f(\gm)$. Therefore, $d_\gm(\af+\bt)=d_\gm(0)$. Then \cref{one-half-der-in-terms-of-d_af} takes the form $-4f(\gm)d_\gm(0)=-f(\gm)d_\gm(0)$, whence $d_\gm(0)=0$.
\end{proof}

\begin{prop}\label{half-der-|f(G)|=3}
	Let $|f(\G)|=3$. Then $\frac 12$-derivations of $V(f)$ are of the form
	\begin{align}
		\vf(e_\af)=\sum_{\gm\in\G_0} a_\gm e_{\af+\gm},
	\end{align} 
	where $\{a_\gm\}_{\gm\in\G_0}\sst\C$ is a sequence having only a finite number of non-zero elements.
\end{prop} 
\begin{proof}
	A consequence of \cref{d_gm(af)=d_gm(0)-if-gm-in-G_0,d_gm(af)=d_gm(0)-for-|f(G)|=3,vf(e_af)=ae_(af+gm)-gm-in-G_0}.
\end{proof}

\subsubsection{The case $|f(\G)|=2$}
We have $f(\gm)=0$ for all $\gm\in\G_0$ and $f(\gm)=c\in\C^*$ for all $\gm\in\G\setminus\G_0$.

\begin{lem}\label{d_gm(af)=d_gm(0)-for-|f(G)|=2}
	Let $|f(\G)|=2$, $\gm\not\in\G_0$ and $\vf_\gm\in\Dl(V(f))$ satisfying \cref{vf_af(e_bt)=d_af(bt)e_(af+bt)}. Then 
	\begin{align*}
		d_\gm(\af)=
		\begin{cases}
			d_\gm(0), & \af\in\G_0,\\
			0, & \af\not\in\G_0.
		\end{cases}
	\end{align*}
\end{lem}
\begin{proof}
	If $\af\in\G_0$, then $0=f(\af)\ne f(\gm)=c$, so $d_\gm(\af)=d_\gm(0)$ by \cref{d_gm(af)=d_gm(0)}.
	
	If $\af\not\in\G_0$, then $f(\af)=f(\gm)=c$. Either $\af+\gm\in\G_0$, in which case $f(\af+\gm)=0\ne 2c=2f(\gm)$, or $\af+\gm\not\in\G_0$, in which case $f(\af+\gm)=c\ne 2c=2f(\gm)$, so $d_\gm(\af)=0$ by \cref{d_gm(af)=d_gm(0)}. 
\end{proof}

\begin{lem}\label{vf(e_af)=ae_(af+gm)-|f(G)|=2}
	Let $|f(\G)|=2$. Then for any fixed $\gm\not\in\G_0$ and $a\in\C$ the linear map $\vf:V(f)\to V(f)$ given by
	\begin{align*}
		\vf(e_\af)=
		\begin{cases}
			ae_{\af+\gm}, & \af\in\G_0,\\
			0, & \af\not\in\G_0,
		\end{cases}
	\end{align*}
	is a $\frac 12$-derivation of $V(f)$.
\end{lem}
\begin{proof}
	We have $\vf=\vf_\gm$, where 
	\begin{align*}
		d_\gm(\af)=
			\begin{cases}
			a, & \af\in\G_0,\\
			0, & \af\not\in\G_0,
		\end{cases}
	\end{align*}
for all $\af\in\G$. Take arbitrary $\af,\bt\in\G$. 
	
	\textit{Case 1.} $\af,\bt\in\G_0$. Then  $\af+\gm,\bt+\gm\not\in\G_0$. Hence, $f(\af)=f(\bt)=0$, $f(\af+\gm)=f(\bt+\gm)$ and $d_\gm(\af)=d_\gm(\bt)=a$, so both sides of \cref{one-half-der-in-terms-of-d_af} are zero.
	
	\textit{Case 2.} $\af\in\G_0$, $\bt\not\in\G_0$. Then $\af+\bt,\af+\gm\not\in\G_0$, so $d_\gm(\af+\bt)=d_\gm(\bt)=0$ and $f(\af+\gm)=f(\bt)$. Again, both sides of \cref{one-half-der-in-terms-of-d_af} are equal to $0$.
	
	\textit{Case 3.} $\af\not\in\G_0$, $\bt\in\G_0$. Similarly to Case~2, we have $d_\gm(\af+\bt)=d_\gm(\af)=0$ and $f(\bt+\gm)=f(\af)$, so both sides of \cref{one-half-der-in-terms-of-d_af} are equal to $0$.
	
	\textit{Case 4.} $\af,\bt\not\in\G_0$. Then $f(\af)=f(\bt)$ and $d_\gm(\af)=d_\gm(\bt)=0$, so both sides of \cref{one-half-der-in-terms-of-d_af} are again zero.
	%
	%
\end{proof}

\begin{prop}\label{half-der-|f(G)|=2}
	Let $|f(\G)|=2$. Then $\frac 12$-derivations of $V(f)$ are of the form
	\begin{align}
		\vf(e_\af)=
		\begin{cases}
			\sum_{\gm\in\G} a_\gm e_{\af+\gm}, & \af\in\G_0,\\
			\sum_{\gm\in\G_0} a_\gm e_{\af+\gm}, & \af\not\in\G_0,
		\end{cases}
	\end{align} 
	where $\{a_\gm\}_{\gm\in\G}\sst\C$ is a sequence having only a finite number of non-zero elements.
\end{prop} 
\begin{proof}
	A consequence of \cref{d_gm(af)=d_gm(0)-if-gm-in-G_0,d_gm(af)=d_gm(0)-for-|f(G)|=2,vf(e_af)=ae_(af+gm)-gm-in-G_0,vf(e_af)=ae_(af+gm)-|f(G)|=2}.
\end{proof}

Based on \cite[Theorem 1]{fil1} and \cref{half-der-|f(G)|>=4,half-der-|f(G)|=3,half-der-|f(G)|=2}, we have the following corollary.

 \begin{cor}
The Witt type Lie algebra $V(f)$ admits a non-trivial ${\rm Hom}$-Lie algebra structure.
 \end{cor}

\subsection{Transposed Poisson structures on $V(f)$}\label{sec-tp}

\subsubsection{The case $|f(\G)|\ge 4$}

\begin{lem}\label{e_af-cdot-e_bt=e_af+bt-tp-|f(G)|>=4}
	Let $|f(\G)|\ge 4$. Then $e_\af\cdot e_\bt=e_{\af+\bt}$ defines a transposed Poisson algebra structure on $V(f)$.
\end{lem}
\begin{proof}
	Fix $\bt\in\G$ and let $\vf(e_\af)=e_{\af+\bt}$. Then $\vf\in\Dl(V(f))$ by \cref{half-der-|f(G)|>=4}, where the corresponding sequence $\{a_\gm\}_{\gm\in\G}\sst\C$ has $a_\bt=1$ and $a_\gm=0$ for $\gm\ne\bt$. Since $\cdot$ is clearly commutative and associative, \cref{glavlem} guarantees that $\cdot$ is a transposed Poisson algebra structure on $V(f)$.
\end{proof}

As a consequence of  \cref{e_af-cdot-e_bt=e_af+bt-tp-|f(G)|>=4} and \cite[Lemma 5]{FKL} we obtain the following.
\begin{cor}\label{mut-e_af-cdot-e_bt=e_af+bt-tp-|f(G)|>=4}
	Let $|f(\G)|\ge 4$. Then every mutation $e_\af\cdot_b e_\bt:=e_\af\cdot b\cdot e_\bt$ of $e_\af\cdot e_\bt=e_{\af+\bt}$ is a transposed Poisson algebra structure on $V(f)$.
\end{cor}

\begin{lem}\label{tp-mut-e_af-cdot-e_bt=e_af+bt-|f(G)|>=4}
	Let $|f(\G)|\ge 4$. Then every transposed Poisson algebra structure on $V(f)$ is a mutation of $e_\af\cdot e_\bt=e_{\af+\bt}$, where $\af,\bt\in\G$. 
\end{lem}
\begin{proof}
	Let $*$ be a Poisson algebra structure on $V(f)$ and $\af,\bt\in\G$. By \cref{glavlem,half-der-|f(G)|>=4} there exist sequences $\{a^\af_\gm\}_{\gm\in\G}\sst\C$ and $ \{a^\bt_\gm\}_{\gm\in\G}\sst\C$ with only a finite number of non-zero elements such that
	\begin{align*}
		e_\af*e_\bt=\sum_{\gm\in\G} a^\af_\gm e_{\bt+\gm}\text{ and }e_\bt*e_\af=\sum_{\gm\in\G} a^\bt_\gm e_{\af+\gm}.
	\end{align*} 
	Since $*$ is commutative, for any $\gm,\gm'\in\G$ if $\bt+\gm=\af+\gm'$, then $a^\af_{\gm}=a^\bt_{\gm'}$. Hence,
	\begin{align}\label{a^af_gm=a^bt_(bt-af+gm)}
		a^\af_{\gm}=a^\bt_{\bt-\af+\gm}
	\end{align}
for all $\gm\in\G$. Define $b_\gm=a^0_\gm$, so that $a_\gm^\af=b_{\gm-\af}$ by \cref{a^af_gm=a^bt_(bt-af+gm)}. Then
\begin{align*}
	e_\af*e_\bt=\sum_{\gm\in\G} b_{\gm-\af} e_{\bt+\gm}=\sum_{\gm\in\G} b_\gm e_{\af+\bt+\gm}=e_\af\cdot\left(\sum_{\gm\in\G} b_\gm e_\gm\right)\cdot e_\bt=e_\af\cdot_b e_\bt,
\end{align*}
where $b=\sum_{\gm\in\G} b_\gm e_\gm$.
\end{proof}

\begin{prop}\label{tp-4}
	Let $|f(\G)|\ge 4$. Then transposed Poisson algebra structures on $V(f)$ are exactly mutations of $e_\af\cdot e_\bt=e_{\af+\bt}$, where $\af,\bt\in\G$. 
\end{prop}
\begin{proof}
	A consequence of \cref{mut-e_af-cdot-e_bt=e_af+bt-tp-|f(G)|>=4,tp-mut-e_af-cdot-e_bt=e_af+bt-|f(G)|>=4}.
\end{proof}

\subsubsection{The case $|f(\G)|=3$}

Since $\G_0$ is the kernel of an epimorphism $\tau:\G\to\Z/3\Z$, we have $|\G/\G_0|=3$. Let us fix $\gm_i\in\G$, $i=0,1,2$, such that $\G=\sqcup_{i=0}^2(\gm_i+\G_0)$, where $\gm_0=0$. We thus have the direct sum of vector spaces
\begin{align}\label{V(f)=direct-sum-V(f)_i}
	V(f)=\bigoplus_{i=0}^2 V(f)_i,
\end{align}
where $V(f)_i$ is spanned by $\{e_{\gm_i+\gm}\}_{\gm\in\G_0}$. As above, we denote by $\cdot$ the multiplication $e_\af\cdot e_\bt=e_{\af+\bt}$ on $V(f)$. Observe that $V(f)_i\cdot V(f)_j\sst V(f)_k $, where $k\equiv(i+j)\,\mathrm{mod}\,3$, so \cref{V(f)=direct-sum-V(f)_i} is a $\Z_3$-grading of $(V(f),\cdot)$. In particular, $V(f)_i$ is not closed under $\cdot$ for $i>0$, but it turns out to be closed under a mutation of $\cdot$.

\begin{lem}
	Let $|f(\G)|=3$, $i\in\{0,1,2\}$ and $b_i\in V(f)_j$, where $j\in\{0,1,2\}$, $i+j\equiv 0 \,\mathrm{mod}\,3$. Then $V(f)_i$ is a commutative and associative algebra under $e_\af\cdot_{b_i} e_\bt:=e_\af\cdot b_i\cdot e_\bt$, where $\af,\bt\in\gm_i+\G_0$.
\end{lem}
\begin{proof}
	It suffices to prove that $V(f)_i$ is closed under $\cdot_b$. This is because $V(f)_i\cdot_{b_i} V(f)_i\sst V(f)_i\cdot V(f)_j\cdot V(f)_i\sst V(f)_k$, where $k\equiv 2i+j\equiv i\,\mathrm{mod}\,3$.
\end{proof}

\begin{lem}\label{direct-sum-tp-|f(G)|=3}
	Let $|f(\G)|=3$, $i\in\{0,1,2\}$ and $b_i\in V(f)_j$, where $j\in\{0,1,2\}$, $i+j\equiv 0 \,\mathrm{mod}\,3$. Then $\bigoplus_{i=0}^2 (V(f)_i,\cdot_{b_i})$ is a transposed Poisson algebra structure on $V(f)$.
\end{lem}
\begin{proof}
	Clearly, $\bigoplus_{i=0}^2 (V(f)_i,\cdot_{b_i})$ is commutative and associative. Fix $\bt\in\G$ and let $i\in\{0,1,2\}$ such that $\bt\in\gm_i+\G_0$. Denote by $\vf$ the multiplication by $e_\bt$ in $\bigoplus_{i=0}^2 (V(f)_i,\cdot_{b_i})$. Then 
	\begin{align*}
		\vf(e_\af)=
		\begin{cases}
			0, & \af\not\in\gm_i+\G_0,\\
			e_\af\cdot_{b_i}e_\bt, & \af\in\gm_i+\G_0.
		\end{cases}
	\end{align*}
Write $b_i=\sum_{\gm\in\G_0}(b_i)_\gm e_{\gm-\gm_i}$. Then $e_\af\cdot_{b_i}e_\bt=\sum_{\gm\in\G_0}(b_i)_\gm e_{\af+\bt+\gm-\gm_i}$. Since $\bt+\gm-\gm_i=(\bt-\gm_i)+\gm\in\G_0$, the latter sum can be rewritten as $\sum_{\gm\in\G_0}(b_i)_{\gm+\gm_i-\bt} e_{\af+\gm}$. Define
\begin{align*}
	a_\gm=
	\begin{cases}
		0, & \af\not\in\gm_i+\G_0,\\
		(b_i)_{\gm+\gm_i-\bt}, & \af\in\gm_i+\G_0.
	\end{cases}
\end{align*}
Then $\vf(e_\af)=\sum_{\gm\in\G_0}a_\gm e_{\af+\gm}$, so
$\vf\in\Dl(V(f))$ by \cref{half-der-|f(G)|=3}. The result now follows from \cref{glavlem}. 
\end{proof}

\begin{lem}\label{tp-is-direct-sum-|f(G)|=3}
	Let $|f(\G)|=3$. Then every transposed Poisson algebra structure on $V(f)$ is of the form $\bigoplus_{i=0}^2 (V(f)_i,\cdot_{b_i})$ for some $b_i\in V(f)_j$ with $i+j\equiv 0 \,\mathrm{mod}\,3$. 
\end{lem}
\begin{proof}
	Let $*$ be a Poisson algebra structure on $V(f)$ and $\af,\bt\in\G$. By \cref{glavlem,half-der-|f(G)|=3} there are sequences $\{a^\af_\gm\}_{\gm\in\G_0}\sst\C$ and $ \{a^\bt_\gm\}_{\gm\in\G_0}\sst\C$ with only a finite number of non-zero elements such that
	\begin{align*}
		e_\af*e_\bt=\sum_{\gm\in\G_0} a^\af_\gm e_{\bt+\gm}\text{ and }e_\bt*e_\af=\sum_{\gm\in\G_0} a^\bt_\gm e_{\af+\gm}.
	\end{align*}

\textit{Case 1.} $\af-\bt\not\in\G_0$. Assume that $e_\af*e_\bt\ne 0$, so there exists $\gm\in\G_0$ with $a^\af_\gm\ne 0$. Since $e_\af*e_\bt=e_\bt*e_\af$, there also exists $\gm'\in\G_0$ such that $\bt+\gm=\af+\gm'$ and $a^\af_\gm=a^\bt_{\gm'}$. But then $\af-\bt=\gm-\gm'\in\G_0$, a contradiction. Thus, $e_\af*e_\bt=0$.

\textit{Case 2.} $\af-\bt\in\G_0$. Let $i\in\{0,1,2\}$ such that $\af,\bt\in\gm_i+\G_0$. Given $\gm,\gm'\in\G_0$ with $\bt+\gm=\af+\gm'$, the commutativity of $*$ implies $a^\af_{\gm}=a^\bt_{\gm'}$, so $a^\af_{\gm}=a^\bt_{\bt-\af+\gm}$ for all $\gm\in\G_0$. Define $b_i=\sum_{\gm\in\G_0}a^{\gm_i}_\gm e_{\gm-\gm_i}$. Then $b_i\in V(f)_j$ with $i+j\equiv 0 \,\mathrm{mod}\,3$ and
\begin{align*}
	e_\af*e_\bt=\sum_{\gm\in\G_0} a^\af_\gm e_{\bt+\gm}=\sum_{\gm\in\G_0} a^{\gm_i}_{(\gm_i-\af)+\gm} e_{\bt+\gm}=\sum_{\gm\in\G_0} a^{\gm_i}_{\gm} e_{\bt+\gm-(\gm_i-\af)}=\sum_{\gm\in\G_0} a^{\gm_i}_{\gm} e_{\bt+\af+\gm-\gm_i}=e_\af\cdot_{b_i}\cdot e_\bt.
\end{align*}
Thus, $(V(f),*)=\bigoplus_{i=0}^2 (V(f)_i,\cdot_{b_i})$.
\end{proof}

\begin{prop}\label{tp-3}
	Let $|f(\G)|=3$. Then transposed Poisson algebra structures on $V(f)$ are exactly of the form $\bigoplus_{i=0}^2 (V(f)_i,\cdot_{b_i})$ for some $b_i\in V(f)_j$ with $i+j\equiv 0 \,\mathrm{mod}\,3$.  
\end{prop}
\begin{proof}
	A consequence of \cref{direct-sum-tp-|f(G)|=3,tp-is-direct-sum-|f(G)|=3}.
\end{proof}

\subsubsection{The case $|f(\G)|=2$}

We again denote by $\cdot$ the multiplication $e_\af\cdot e_\bt=e_{\af+\bt}$ on $V(f)$. For any $b=\sum_{\gm\in\G}b_\gm e_\gm\in V(f)$ denote $b^0=\sum_{\gm\in\G_0}b_\gm e_\gm\in V(f)$.

\begin{lem}\label{e_af*a_bt-tp}
	Let $|f(\G)|=2$ and $b\in V(f)$. Then
	\begin{align}\label{e_af*a_bt-|f(G)|=2}
		e_\af*e_\bt=
		\begin{cases}
			e_\af\cdot_b e_\bt, & \af,\bt\in\G_0,\\
			e_\af\cdot_{b^0} e_\bt, & \af\in\G_0,\bt\not\in\G_0\text{ or }\af\not\in\G_0,\bt\in\G_0,\\
			0, & \af,\bt\not\in\G_0.
		\end{cases}
	\end{align}
defines a transposed Poisson algebra structure on $V(f)$.
\end{lem}
\begin{proof}
	It is obvious that $*$ is commutative. Let us check associativity of $*$. Take arbitrary $\af,\bt,\gm\in\G$. 
	
	\textit{Case 1.} $\af,\bt,\gm\in\G_0$. Write $b=b^0+b^1$. Then
	\begin{align*}
		(e_\af*e_\bt)*e_\gm&=(e_\af\cdot_b e_\bt)*e_\gm=(e_\af\cdot_{b^0} e_\bt)*e_\gm+(e_\af\cdot_{b^1} e_\bt)*e_\gm\\
		&=(e_\af\cdot_{b^0} e_\bt)\cdot_b e_\gm+(e_\af\cdot_{b^1} e_\bt)\cdot_{b^0} e_\gm=e_\af\cdot b^0\cdot e_\bt\cdot b\cdot e_\gm+e_\af\cdot b^1\cdot e_\bt\cdot b^0\cdot e_\gm\\
		&=e_{\af+\bt+\gm}\cdot b^0\cdot (b+b^1).
	\end{align*}
	 and
	 \begin{align*}
	 	e_\af*(e_\bt*e_\gm)&=e_\af*(e_\bt\cdot_b e_\gm)=e_\af*(e_\bt\cdot_{b^0} e_\gm)+e_\af*(e_\bt\cdot_{b^1} e_\gm)\\
	 	&=e_\af\cdot_b(e_\bt\cdot_{b^0} e_\gm)+e_\af\cdot_{b^0}(e_\bt\cdot_{b^1} e_\gm)=e_\af\cdot b\cdot e_\bt\cdot b^0\cdot e_\gm+e_\af\cdot b^0\cdot e_\bt\cdot b^1\cdot e_\gm\\
        &=e_{\af+\bt+\gm}\cdot b^0\cdot (b+b^1).
	 \end{align*}
 Thus, $(e_\af*e_\bt)*e_\gm=e_\af*(e_\bt*e_\gm)$.
 
 \textit{Case 2.} $\af,\bt\in\G_0$ and $\gm\not\in\G_0$. Then
 \begin{align*}
 	(e_\af*e_\bt)*e_\gm&=(e_\af\cdot_b e_\bt)*e_\gm=(e_\af\cdot_{b^0} e_\bt)*e_\gm+(e_\af\cdot_{b^1} e_\bt)*e_\gm\\
 	&=(e_\af\cdot_{b^0} e_\bt)*e_\gm=(e_\af\cdot_{b^0} e_\bt)\cdot_{b^0} e_\gm.
 \end{align*}
 and
 \begin{align*}
 	e_\af*(e_\bt*e_\gm)=e_\af*(e_\bt\cdot_{b^0} e_\gm)=e_\af\cdot_{b^0}(e_\bt\cdot_{b^0} e_\gm).
 \end{align*}
Since $\cdot_{b^0}$ is associative, we have $(e_\af*e_\bt)*e_\gm=e_\af*(e_\bt*e_\gm)$.

\textit{Case 3.} $\af,\gm\in\G_0$ and $\bt\not\in\G_0$. Then $(e_\af*e_\bt)*e_\gm=e_\gm*(e_\af*e_\bt)=(e_\gm*e_\af)*e_\bt$ by commutativity of $*$ and the result of Case 2. Similarly,
$e_\af*(e_\bt*e_\gm)=e_\af*(e_\gm*e_\bt)=(e_\af*e_\gm)*e_\bt=(e_\gm*e_\af)*e_\bt$. Thus, $(e_\af*e_\bt)*e_\gm=e_\af*(e_\bt*e_\gm)$.

\textit{Case 4.} $\bt,\gm\in\G_0$ and $\af\not\in\G_0$. Then $(e_\af*e_\bt)*e_\gm=e_\gm*(e_\af*e_\bt)=e_\gm*(e_\bt*e_\af)$ and $e_\af*(e_\bt*e_\gm)=(e_\bt*e_\gm)*e_\af=(e_\gm*e_\bt)*e_\af$, so we are in the conditions of Case 2.

\textit{Case 5.} $\af,\bt\not\in\G_0$ and $\gm\in\G_0$. Then $e_\af*e_\bt=0$, so $(e_\af*e_\bt)*e_\gm=0$. Furthermore, $e_\af*(e_\bt*e_\gm)=e_\af*(e_\bt\cdot_{b^0} e_\gm)=0$.

\textit{Case 6.} $\af,\gm\not\in\G_0$ and $\bt\in\G_0$. This case reduces to Case 5 the same way as Case 3 reduces to Case 2.

\textit{Case 7.} $\bt,\gm\not\in\G_0$ and $\af\in\G_0$. This case reduces to Case 5 the same way as Case 4 reduces to Case 2.

\textit{Case 8.} $\af,\bt,\gm\not\in\G_0$. Then $e_\af*e_\bt=e_\bt*e_\gm=0$, so $(e_\af*e_\bt)*e_\gm=0=e_\af*(e_\bt*e_\gm)$.
\end{proof}

\begin{lem}\label{tp-|f(G)|=2}
	Let $|f(\G)|=2$. Then every transposed Poisson algebra structure on $V(f)$ is of the form \cref{e_af*a_bt-|f(G)|=2}.
\end{lem}
\begin{proof}
	By \cref{glavlem,half-der-|f(G)|=2} there are sequences $\{a^\af_\gm\}_{\gm\in\G}\sst\C$ and $ \{a^\bt_\gm\}_{\gm\in\G}\sst\C$ with only a finite number of non-zero elements such that
	\begin{align*}
		e_\af*e_\bt=
		\begin{cases}
			\sum_{\gm\in\G} a^\af_\gm e_{\bt+\gm}, & \bt\in\G_0,\\
			\sum_{\gm\in\G_0} a^\af_\gm e_{\bt+\gm}, & \bt\not\in\G_0,
		\end{cases}
	\text{ and }
	e_\bt*e_\af=
	\begin{cases}
		\sum_{\gm\in\G} a^\bt_\gm e_{\af+\gm}, & \af\in\G_0,\\
		\sum_{\gm\in\G_0} a^\bt_\gm e_{\af+\gm}, & \af\not\in\G_0.
	\end{cases}
	\end{align*}
	
	\textit{Case 1.} $\af,\bt\in\G_0$. By commutativity of $*$ we have $\sum_{\gm\in\G} a^\af_\gm e_{\bt+\gm}=\sum_{\gm\in\G} a^\bt_\gm e_{\af+\gm}$, whence $a^\af_{\gm}=a^\bt_{\bt-\af+\gm}$. In particular, $a^\af_{\gm}=a^0_{\gm-\af}$. So, as in \cref{tp-mut-e_af-cdot-e_bt=e_af+bt-|f(G)|>=4}, we set $b_\gm:=a^0_\gm$ and
	\begin{align*}
		e_\af*e_\bt=\sum_{\gm\in\G} b_{\gm-\af} e_{\bt+\gm}=\sum_{\gm\in\G} b_\gm e_{\af+\bt+\gm}=e_\af\cdot\left(\sum_{\gm\in\G} b_\gm e_\gm\right)\cdot e_\bt=e_\af\cdot_b e_\bt,
	\end{align*}
where $b=\sum_{\gm\in\G} b_\gm e_\gm$.

	\textit{Case 2.} $\af\in\G_0$, $\bt\not\in\G_0$. Then $e_\af*e_\bt=\sum_{\gm\in\G_0} a^\af_\gm e_{\bt+\gm}$. As in Case 1 we have $a^\af_{\gm}=a^0_{\gm-\af}=b_{\gm-\af}$, so
	\begin{align*}
		e_\af*e_\bt=\sum_{\gm\in\G_0} b_{\gm-\af} e_{\bt+\gm}=\sum_{\gm\in\G_0} b_\gm e_{\af+\bt+\gm}=e_\af\cdot\left(\sum_{\gm\in\G_0} b_\gm e_\gm\right)\cdot e_\bt=e_\af\cdot_{b^0} e_\bt.
	\end{align*}
	Observe that $\bt+\gm\in\bt+\G_0$ for all $\gm\in\G_0$. Since $e_\bt*e_\af=\sum_{\gm\in\G} a^\bt_\gm e_{\af+\gm}$, it follows from $e_\af*e_\bt=e_\bt*e_\af$ that $a^\bt_\gm=0$ whenever $\af+\gm\not\in \bt+\G_0$. But $\af\in\G_0$, so the latter is equivalent to $\gm\not\in\bt+\G_0$. And if $\gm\in\bt+\G_0$, then $a^\bt_\gm=a^\af_{\gm-\bt+\af}=b_{\gm-\bt}$. Thus,
	\begin{align}\label{a^bt_af=0-or-b_gm-bt}
		a^\bt_\gm=
		\begin{cases}
			0, & \gm\not\in \bt+\G_0,\\
			b_{\gm-\bt}, &\gm\in \bt+\G_0.
		\end{cases}
	\end{align}
	
	\textit{Case 3.} $\af\not\in\G_0$, $\bt\in\G_0$. By the result of Case 2 we have $e_\af*e_\bt=e_\bt*e_\af=e_\af\cdot_{b^0} e_\bt$. 
	
	\textit{Case 4.} $\af,\bt\not\in\G_0$. Then $e_\af*e_\bt=\sum_{\gm\in\G_0} a^\af_\gm e_{\bt+\gm}$. But $\gm\in\G_0$ implies $\gm\not\in \af+\G_0$ , so $a^\af_\gm=0$ for all $\gm\in\G_0$ by \cref{a^bt_af=0-or-b_gm-bt}. Thus, $e_\af*e_\bt=0$.
\end{proof}

\begin{prop}\label{tp-2}
	Let $|f(\G)|=2$. Then transposed Poisson algebra structures on $V(f)$ are exactly of the form \cref{e_af*a_bt-|f(G)|=2} for some $b\in V(f)$.  
\end{prop}
\begin{proof}
	A consequence of \cref{e_af*a_bt-tp,tp-|f(G)|=2}.
\end{proof}

\medskip


\end{document}